\documentclass[12pt]{article}
\usepackage{amsfonts}
\usepackage{amssymb}
\usepackage{amsmath}

\newtheorem{theorem}{Theorem}[section]

\newtheorem{corollary}{Corollary}[section]

\newtheorem{definition}{Definition}[section]
\newtheorem{example}{Example}[section]

\newtheorem{lemma}{Lemma}[section]

\newtheorem{remark}{Remark}[section]

\newenvironment{proof}[1][Proof]{\noindent\textbf{#1.} }{\ \rule{0.5em}{0.5em}}
\input{tcilatex}
\begin{document}

\title{Permutative nonnegative matrices with prescribed spectrum\thanks{%
Supported by Universidad Cat\'{o}lica del Norte, Chile. }}
\author{ Ricardo L. Soto\thanks{%
E-mail addresses: rsoto@ucn.cl (R.L. Soto)} \\
{\small Dpto. Matem\'{a}ticas, Universidad Cat\'{o}lica del Norte, Casilla
1280}\\
{\small Antofagasta, Chile.}}
\date{}
\maketitle

\begin{abstract}
An $n\times n$ permutative matrix is a matrix in which every row is a
permutation of the first row. In this paper the result given by Paparella in
[Electron. J. Linear Algebra 31 (2016) 306-312] is extended to a more
general lists of real and complex numbers, and a negative answer to a
question posed by him is given.
\end{abstract}



\section{Introduction}

\noindent The \textit{nonnegative inverse eigenvalue problem }(NIEP) is the
problem of characterizing all posible spectra of entrywise nonnegative
matrices. This problem remain unsolved. A complete solution is known only
for $n\leq 4.$ A list $\Lambda =\{\lambda _{1},\lambda _{2},\ldots ,\lambda
_{n}\}$ of complex numbers is said to be realizable if $\Lambda $ is the
spectrum of an $n\times n$ nonnegative matrix $A$. In this case $A$ \ is
said to be a realizing matrix. From the Perron-Frobenius theory, we have
that if $\{\lambda _{1},\lambda _{2},\ldots ,\lambda _{n}\}$ is the spectrum
of an $n\times n$ nonnegative matrix $A$ then $\rho (A)=\max\limits_{1\leq
i\leq n}\left\vert \lambda _{i}\right\vert $ is an eigenvalue of $A.$ This
eigenvalue is called the \textit{Perron eigenvalue} of $A$ and we shall
assume, in this paper, that $\rho (A)=\lambda _{1}.$ A matrix $A=[a_{ij}]$
is said to have \textit{constant row sums }if all its rows sum up to the
same constant, say $\alpha ,$ i.e., $\dsum\limits_{j=1}^{n}a_{ij}=\alpha ,$ $%
i=1,\ldots ,n.$\ The set of all matrices with constant row sums equal to $%
\alpha $ will be denoted by $\mathcal{CS}_{\alpha }.$\emph{\ }It is clear
that any matrix in\emph{\ }$\mathcal{CS}_{\alpha }$ has the eigenvector $%
\mathbf{e}=(1,1,\ldots ,1)^{T}$ corresponding to the eigenvalue $\alpha .$
We shall denote by $\mathbf{e}_{k}$ the n-dimensional vector with one in the 
$kth$ position and zeros elsewhere. The real matrices with constant row sums
are important because it is known that the problem of finding a nonnegative
matrix with spectrum $\Lambda =\{\lambda _{1},\ldots ,\lambda _{n}\}$ is
equivalent to the problem of finding a nonnegative matrix in \emph{\ }$%
\mathcal{CS}_{\lambda _{1}}$ with spectrum $\Lambda .$

\bigskip

\noindent The following definition, given in \cite{Paparella}, is due to
Charles Johnson:

\begin{definition}
Let $\mathbf{x}\in 
\mathbb{C}
^{n}$ and let $P_{2},\ldots ,P_{n}$ be $n\times n$ permutation matrices. A
permutative matrix is any matrix of the form%
\begin{equation*}
P=\left[ 
\begin{array}{c}
\mathbf{x}^{T} \\ 
(P_{2}\mathbf{x})^{T} \\ 
\vdots \\ 
(P_{n}\mathbf{x})^{T}%
\end{array}%
\right] .
\end{equation*}
\end{definition}

\noindent It is clear that $P\in \mathcal{CS}_{S},$ where $S$ is the sum of
the entries of the vector $\mathbf{x}.$ In \cite{Paparella}, the author
prove that a list $\Lambda =\{\lambda _{1},\ldots ,\lambda _{n}\}$ of real
numbers of Suleimanova type \cite{Sule}, that is $\lambda _{1}>0\geq \lambda
_{2}\geq \ldots \geq \lambda _{n},$ is realizable by a permutative
nonnegative matrix. The author in \cite{Paparella} also pose the question:
can all relizable lists of real numbers be realized by a permutative
nonnegative matrix? The following result was announced by Suleimanova \cite%
{Sule} and proved by Perfect \cite{Perfect}.

\begin{theorem}
\label{Sulei}Let $\Lambda =\{\lambda _{1},\lambda _{2},\ldots ,\lambda
_{n}\} $ be a list of real numbers with $\lambda _{i}<0,$ $i=2,\ldots ,n.$
Then $\Lambda $ is the spectrum of an $n\times n$ nonnegative matrix if and
only if $\dsum\limits_{i=1}^{n}\lambda _{i}\geq 0.$
\end{theorem}

\noindent The following two results, which we set here for completeness,
have been shown to be very useful, not only to derive sufficient conditions
for the realizability of the \textit{NIEP, }but for constructing a realizing
matrix as well. The first result, due to Brauer \cite{Brauer}, shows how to
modify one single eigenvalue of a matrix, via a rank-one perturbation,
without changing any of its remaining eigenvalues (see \cite{Perfect, Soto,
Soto3} and the references therein to see how Brauer result has been applied
to \textit{NIEP)}. The second result, due to Rado and introduced by Perfect
in \cite{Perfect1} is an extension of Brauer result and it shows how to
change $r$ eigenvalues of an $n\times n$ matrix ($r<n),$ via a perturbation
of rank $r,$ without changing any of its remaining $n-r$ eigenvalues (see 
\cite{Perfect1, Soto1} to see how Rado result has been applied to \textit{%
NIEP}). Both results will be employed here to obtain conditions for lists of
real and complex numbers to be the spectrum of a permutative nonnegative
matrix.

\begin{theorem}
Brauer \cite{Brauer}\label{Bra} Let $A$ be an $n\times n$ arbitrary matrix
with eigenvalues $\lambda _{1},\lambda _{2},\ldots ,\lambda _{n}.$ Let $%
\mathbf{v}=(v_{1},\ldots ,v_{n})^{T}$ be an eigenvector of $A$ corresponding
to the eigenvalue $\lambda _{k}$ and let $\mathbf{q}$ be any $n-$dimensional
vector. Then the matrix $A+\mathbf{vq}^{T}$ has eigenvalues $\lambda
_{1},\ldots \lambda _{k-1},\lambda _{k}+\mathbf{v}^{T}\mathbf{q,}\lambda
_{k+1},\ldots ,\lambda _{n}.$
\end{theorem}

\begin{theorem}
Rado \cite{Perfect1}\label{Rado} Let $A$ be an $n\times n$ arbitrary matrix
with spectrum $\Lambda =\{\lambda _{1},\ldots ,\lambda _{n}\}.$ Let $X=\left[
\mathbf{x}_{1}\mid \cdots \mid \mathbf{x}_{r}\right] $ be such that $%
rank(X)=r$ and $A\mathbf{x}_{i}=\lambda _{i}\mathbf{x}_{i},$ $i=1,\ldots ,r,$
$r\leq n.$ Let $C$ be an $r\times n$ arbitrary matrix. Then $A+XC$ has
eigenvalues $\mu _{1},\ldots ,\mu _{r},\lambda _{r+1},\ldots \lambda _{n},$
where $\mu _{1},\ldots ,\mu _{r}$ are eigenvalues of the matrix $\Omega +CX$
with $\Omega =diag\{\lambda _{1},\ldots ,\lambda _{r}\}.$
\end{theorem}

\noindent A simple proof of Theorem \ref{Sulei} was given in \cite{Soto} by
applying Brauer result. The following result in \cite{Soto2}, is a symmetric
version of the Rado result, which we shall use to obtain some of the results
in this paper:

\begin{theorem}
\cite{Soto2}\label{So2} Let $A$ be an $n\times n$ real symmetric matrix with
spectrum $\Lambda =\{\lambda _{1},\lambda _{2},\ldots ,\lambda _{n}\},$ and
for some $r\leq n,$ let $\{\mathbf{x}_{1},\mathbf{x}_{2},\ldots ,\mathbf{x}%
_{r}\}$ be an orthonormal set of eigenvectors of $A$ spanning the invariant
subspace associated with $\lambda _{1},\lambda _{2},\ldots ,\lambda _{r}.$
Let $X$ be the $n\times r$ matrix with $i-th$ column $\mathbf{x}_{i},$ let $%
\Omega =diag\{\lambda _{1},\ldots ,\lambda _{r}\},$ and let $C$ be any $%
r\times r$ symmetric matrix. Then the symmetric matrix $A+XCX^{T}$ has
eigenvalues $\mu _{1},\ldots ,\mu _{r},\lambda _{r+1},\ldots ,\lambda _{n},$
where $\mu _{1},\ldots ,\mu _{r}$ are eigenvalues of the matrix $\Omega +C.$
\end{theorem}

\noindent In this paper we give very simple and short proofs to show that
both, a list of real numbers of Suleimanova type and a list of complex
numbers of Suleimanova type, that is, $\func{Re}\lambda _{k}\leq 0,$ $%
\left\vert \func{Re}\lambda _{k}\right\vert \geq \left\vert \func{Im}\lambda
_{k}\right\vert ,$ $k=2,\ldots ,n,$ are realizable by a permutative
nonnegative matrix. We use theorems \ref{Rado} and \ref{So2} to obtain
sufficient conditions for more general lists to be the spectrum of a
permutative nonnegative matrix and the spectrum of a symmetric permutative
nonnegative matrix. Our results generate an algorithmic procedure to compute
a realizing matrix. The paper is organized as follows: In Section $2$ we
show that a list of real numbers of Suleimanova type is always the spectrum
of a permutative nonnegative matrix, and we give sufficient conditions for
the problem to have a solution in the case of more general lists of real
numbers. We also explore on the existence and construction of symmetric
permutative nonnegative matrices with prescribed spectrum. We show that the
question in \cite{Paparella} has a negative answer, that is, there are
realizable lists of real numbers which are not the spectrum of a permutative
nonnegative matrix. In section $3$ we consider the case of realizable lists
of complex numbers $\Lambda =\{\lambda _{1},\ldots ,\lambda _{n}\}$ of
Suleimanova type, with the condition $\lambda _{n-j+2}=\overline{\lambda _{j}%
},$ \ $j=2,3,\ldots ,\left[ \frac{n+1}{2}\right] ,$ and we show that they
are also realizable by permutative nonnegative matrices. We also give some
examples to illustrate the results.

\section{Permutative matrices with prescribed real spectrum}

\noindent In this section we give a short simple proof of Theorem $3.3$ in 
\cite{Paparella}, and we prove sufficient conditions for the existence of a
(symmetric) permutative nonnegative matrix with prescribed real spectrum. We
also give a response to the question in \cite{Paparella}.

\begin{theorem}
Let $\Lambda =\{\lambda _{1},\lambda _{2},\ldots ,\lambda _{n}\}$ be a a
list of real numbers with $\lambda _{1}>0,$ $\lambda _{i}<0,$ \ $i=2,\ldots
,n.$ Then $\Lambda $ is the spectrum of an $n\times n$ permutative
nonnegative matrix if and only if $\dsum\limits_{i=1}^{n}\lambda _{i}\geq 0.$
\end{theorem}

\begin{proof}
The need is clear. Suppose that $\alpha =\dsum\limits_{i=1}^{n}\lambda
_{i}\geq 0.$ Then we take the list $\Lambda _{\alpha }=\{\lambda _{1}-\alpha
,\lambda _{2},\ldots ,\lambda _{n}\}$\ and consider the initial matrix%
\begin{equation*}
C=\left[ 
\begin{array}{ccccc}
\lambda _{1}-\alpha & 0 & \ddots & \cdots & 0 \\ 
\lambda _{1}-\alpha -\lambda _{2} & \lambda _{2} & \ddots & \ddots & \vdots
\\ 
\lambda _{1}-\alpha -\lambda _{3} & \ddots & \lambda _{3} & \ddots & \vdots
\\ 
\vdots & \ddots & \ddots & \ddots & 0 \\ 
\lambda _{1}-\alpha -\lambda _{n} & 0 & \cdots & 0 & \lambda _{n}%
\end{array}%
\right] \in \mathcal{CS}_{\lambda _{1}-\alpha }.
\end{equation*}%
From the Brauer result, for $\mathbf{q}=\left[ \alpha -\lambda _{1},-\lambda
_{2},-\lambda _{3},\ldots ,-\lambda _{n}\right] ,$ we have that $B=C+\mathbf{%
eq}^{T}$ is a permutative nonnegative matrix with spectrum $\Lambda _{\alpha
},$ and $B\in \mathcal{CS}_{\lambda _{1}-\alpha }.$ Now, $A=B+\mathbf{er}%
^{T},$ where $\mathbf{r}=\left[ \frac{\alpha }{n},\frac{\alpha }{n},\ldots ,%
\frac{\alpha }{n}\right] ,$ is the desired permutative nonnegative matrix
with spectrum $\Lambda $ (of course we may also take $A=C+\mathbf{e(q}^{T}+%
\mathbf{r}^{T})$).
\end{proof}

\bigskip

\noindent Now we give sufficient conditions for more general lists of real
numbers:

\begin{lemma}
\label{Mlanda}The matrix 
\begin{equation}
A=\left[ 
\begin{array}{ccccc}
a_{11} & a_{12} & a_{13} & \cdots & a_{1n} \\ 
a_{11}-\lambda _{2} & a_{12}+\lambda _{2} & a_{13} & \cdots & a_{1n} \\ 
a_{11}-\lambda _{3} & a_{12} & a_{13}+\lambda _{3} & \cdots & a_{1n} \\ 
\vdots & \vdots & a_{13} & \ddots & \vdots \\ 
a_{11}-\lambda _{n} & a_{12} & \cdots & \cdots & a_{1n}+\lambda _{n}%
\end{array}%
\right]  \label{Ma}
\end{equation}%
has eigenvalues $\lambda _{1}=\dsum\limits_{j=1}^{n}a_{1j},\lambda
_{2},\ldots ,\lambda _{n}.$
\end{lemma}

\begin{proof}
Since $A$ has constant row sums equal to $\dsum\limits_{j=1}^{n}a_{1j},$
then $\lambda _{1}=\dsum\limits_{j=1}^{n}a_{1j}.$ Moreover, it is clear that 
$\det (\lambda I-A)=0$ for $\lambda =\lambda _{i},$ $i=2,\ldots ,n.$
\end{proof}

\begin{theorem}
Let $\Lambda =\{\lambda _{1},\lambda _{2},\ldots ,\lambda _{n}\}$ be a list
of real numbers and let $a_{11},a_{12},\ldots ,a_{1n}$ be real nonnegative
numbers. If%
\begin{equation*}
a_{11}=\frac{1}{n}\dsum\limits_{k=1}^{n}\lambda _{k},\text{ \ }%
a_{11}-\lambda _{k}\geq 0,\text{ \ }k=2,\ldots ,n,
\end{equation*}%
then the matrix $A$ in (\ref{Ma}) is permutative nonnegative. If $%
\dsum\limits_{k=1}^{n}\lambda _{k}>0$ and $a_{11}-\lambda _{k}>0,$ then $A$
in (\ref{Ma}) becomes permutative positive.
\end{theorem}

\begin{proof}
It is enough to take $a_{1k}=a_{11}-\lambda _{k},$ $k=2,3,\ldots ,n.$ Then $%
\dsum\limits_{k=1}^{n}a_{1k}=na_{11}-\dsum\limits_{k=2}^{n}\lambda
_{k}=\lambda _{1}$ and $a_{11}=\frac{1}{n}\dsum\limits_{k=1}^{n}\lambda
_{k}. $ Thus the $k^{th}$ row of $A,$ $k=2,\ldots ,n,$ is a permutation of
the first row and $A$ is an $n\times n$ permutative nonnegative matrix with
spectrum $\Lambda .$ It is clear that if $\dsum\limits_{k=1}^{n}\lambda
_{k}>0$ and $a_{11}-\lambda _{k}>0$ then $A$ is positive.
\end{proof}

\begin{corollary}
Let $\Lambda =\{\lambda _{1},\lambda _{2},\ldots ,\lambda _{n}\}$ be a list
of real numbers with $\lambda _{k}<0,$ $k=2,\ldots ,n,$ and $%
\dsum\limits_{k=1}^{n}\lambda _{k}\geq 0.$ Then $\Lambda $ is the spectrum
of a permutative nonnegative (positive) matrix.
\end{corollary}

\begin{proof}
It is enough to take $a_{1k}=a_{11}-\lambda _{k},$ with $a_{11}=\frac{1}{n}%
\dsum\limits_{k=1}^{n}\lambda _{k},$ $k=2,3,\ldots ,n.$ Then $A$ in (\ref{Ma}%
) is permutative nonnegative. If $\dsum\limits_{k=1}^{n}\lambda _{k}>0$ then 
$A$ becomes positive.
\end{proof}

\begin{theorem}
\label{So3}Let $\Lambda =\{\lambda _{1},\lambda _{2},\ldots ,\lambda _{n}\}$
be a list of real numbers. Suppose that: \newline
$i)$ \ There exist a partition $\Lambda =\Lambda _{0}\cup \underset{r\text{
times}}{\underbrace{\Lambda _{1}\cup \cdots \cup \Lambda _{1}}},$ where%
\begin{equation*}
\Lambda _{0}=\{\lambda _{01},\lambda _{02},\ldots ,\lambda _{0r}\},\text{ \ }%
\Lambda _{1}=\{\lambda _{11},\lambda _{12},\ldots ,\lambda _{1p}\},\text{ }
\end{equation*}%
such that $\Gamma _{1}=\{\lambda ,\lambda _{11},\lambda _{12},\ldots
,\lambda _{1p}\},$ $0\leq \lambda \leq \lambda _{1},$ is the spectrum of a $%
(p+1)\times (p+1)$ permutative nonnegative matrix.\newline
$ii)$ There exists an $r\times r$ permutative nonnegative matrix with
spectrum $\Lambda _{0}$ and diagonal entries $\lambda ,\lambda ,\ldots
,\lambda $ $(r$ times$).$\newline
Then, there exists a permutative nonnegative matrix $P$ with spectrum $%
\Lambda .$
\end{theorem}

\begin{proof}
From $i)$ let $P_{1}$ be a $(p+1)\times (p+1)$ permutative nonnegative
matrix with spectrum $\Gamma _{1}$ and let 
\begin{equation*}
A=\left[ \ 
\begin{array}{cccc}
P_{1} &  &  &  \\ 
& P_{1} &  &  \\ 
&  & \ddots &  \\ 
&  &  & P_{1}\ 
\end{array}%
\right] ,
\end{equation*}%
with $r$ blocks $P_{1}.$ Then $P_{1}\mathbf{x}=\lambda \mathbf{x}$ with $%
\mathbf{x}=\frac{1}{\sqrt{p+1}}\mathbf{e}$ (that is $\left\Vert \mathbf{x}%
\right\Vert =1$), where $\lambda $ and $\mathbf{x}$ are the Perron
eigenvalue and the Perron eigenvector of $P_{1}$, respectively.\newline
From $ii)$ let $B$ be the $r\times r$ permutative nonnegative matrix with
spectrum $\Lambda _{0}$ and diagonal entries $\lambda ,\lambda ,\ldots
,\lambda $ $(r$ times$).$ Let $\Omega $ be the $r\times r$ diagonal matrix $%
\Omega =diag\{\lambda ,\lambda ,\ldots ,\lambda \}.$ Then for%
\begin{equation*}
C=B-\Omega ,\text{ \ }X=\left[ \ 
\begin{array}{cccc}
\mathbf{x} & 0 & \cdots & 0 \\ 
0 & \mathbf{x} & \ddots & \vdots \\ 
\vdots & \ddots & \mathbf{\ddots } & 0 \\ 
0 & \ldots & 0 & \mathbf{x}%
\end{array}%
\right] ,
\end{equation*}%
where $X$ is the $r(p+1)\times r$ matrix of eigenvectors of $A,$ it follows
that $XCX^{T}$ is a permutative nonnegative matrix, and from Theorem \ref%
{Rado}, with $C=CX^{T},$ $P=A+XCX^{T}$ is a permutative nonnegative matrix
with spectrum $\Lambda .$ Observe that $P$ is also an $r\times r$ block
permutative nonnegative matrix.
\end{proof}

\begin{example}
Let $\Lambda =\{10,4,2,0,-1,-1,-1,-1,-3,-3,-3,-3\}.$ We take the partition%
\begin{eqnarray*}
\Lambda &=&\Lambda _{0}\cup \Lambda _{1}\cup \Lambda _{1}\cup \Lambda
_{1}\cup \Lambda \text{ \ with} \\
\Lambda _{0} &=&\{10,4,2,0\},\text{ \ }\Lambda _{1}=\{-1,-3\},\text{ \ }%
\Gamma _{1}=\{4,-1,-3\}.
\end{eqnarray*}%
Then we look for a permutative nonnegative matrix $P_{1}$ with spectrum $%
\Gamma _{1},$ and a permutative nonnegative matrix $B$ with spectrum $%
\Lambda _{0}$ and diagonal entries $4,4,4,4.$ These matrices are 
\begin{equation*}
P_{1}=\left[ 
\begin{array}{ccc}
0 & 1 & 3 \\ 
1 & 0 & 3 \\ 
3 & 1 & 0%
\end{array}%
\right] ,\text{ \ }B=\left[ 
\begin{array}{cccc}
4 & 0 & 2 & 4 \\ 
0 & 4 & 2 & 4 \\ 
2 & 0 & 4 & 4 \\ 
4 & 0 & 2 & 4%
\end{array}%
\right] ,
\end{equation*}%
obtained from Theorem \ref{Bra} and Lemma \ref{Mlanda}, respectively. Let 
\begin{equation*}
A=\left[ 
\begin{array}{cccc}
P_{1} &  &  &  \\ 
& P_{1} &  &  \\ 
&  & P_{1} &  \\ 
&  &  & P_{1}%
\end{array}%
\right] ,\text{ \ }X=\left[ 
\begin{array}{cccc}
\mathbf{x} & 0 & 0 & 0 \\ 
0 & \mathbf{x} & 0 & 0 \\ 
0 & 0 & \mathbf{x} & 0 \\ 
0 & 0 & 0 & \mathbf{x}%
\end{array}%
\right] ,\text{ \ with \ }\mathbf{x=}\frac{1}{\sqrt{3}}\mathbf{e.}
\end{equation*}%
Then for $C=B-diag\{4,4,4,4\}$ we have that%
\begin{eqnarray*}
P &=&A+XCX^{T} \\
&=&\left[ 
\begin{array}{cccccccccccc}
0 & 1 & 3 & 0 & 0 & 0 & \frac{2}{3} & \frac{2}{3} & \frac{2}{3} & \frac{4}{3}
& \frac{4}{3} & \frac{4}{3} \\ 
1 & 0 & 3 & 0 & 0 & 0 & \frac{2}{3} & \frac{2}{3} & \frac{2}{3} & \frac{4}{3}
& \frac{4}{3} & \frac{4}{3} \\ 
3 & 1 & 0 & 0 & 0 & 0 & \frac{2}{3} & \frac{2}{3} & \frac{2}{3} & \frac{4}{3}
& \frac{4}{3} & \frac{4}{3} \\ 
0 & 0 & 0 & 0 & 1 & 3 & \frac{2}{3} & \frac{2}{3} & \frac{2}{3} & \frac{4}{3}
& \frac{4}{3} & \frac{4}{3} \\ 
0 & 0 & 0 & 1 & 0 & 3 & \frac{2}{3} & \frac{2}{3} & \frac{2}{3} & \frac{4}{3}
& \frac{4}{3} & \frac{4}{3} \\ 
0 & 0 & 0 & 3 & 1 & 0 & \frac{2}{3} & \frac{2}{3} & \frac{2}{3} & \frac{4}{3}
& \frac{4}{3} & \frac{4}{3} \\ 
\frac{2}{3} & \frac{2}{3} & \frac{2}{3} & 0 & 0 & 0 & 0 & 1 & 3 & \frac{4}{3}
& \frac{4}{3} & \frac{4}{3} \\ 
\frac{2}{3} & \frac{2}{3} & \frac{2}{3} & 0 & 0 & 0 & 1 & 0 & 3 & \frac{4}{3}
& \frac{4}{3} & \frac{4}{3} \\ 
\frac{2}{3} & \frac{2}{3} & \frac{2}{3} & 0 & 0 & 0 & 3 & 1 & 0 & \frac{4}{3}
& \frac{4}{3} & \frac{4}{3} \\ 
\frac{4}{3} & \frac{4}{3} & \frac{4}{3} & 0 & 0 & 0 & \frac{2}{3} & \frac{2}{%
3} & \frac{2}{3} & 0 & 1 & 3 \\ 
\frac{4}{3} & \frac{4}{3} & \frac{4}{3} & 0 & 0 & 0 & \frac{2}{3} & \frac{2}{%
3} & \frac{2}{3} & 1 & 0 & 3 \\ 
\frac{4}{3} & \frac{4}{3} & \frac{4}{3} & 0 & 0 & 0 & \frac{2}{3} & \frac{2}{%
3} & \frac{2}{3} & 3 & 1 & 0%
\end{array}%
\right]
\end{eqnarray*}%
is a permutative nonnegative matrix with spectrum $\Lambda .$ Observe that $%
P $ is also a $4\times 4$ block permutative nonnegative matrix with
permutative blocks.
\end{example}

\begin{remark}
If in the proof of Theorem \ref{So3} the matrices $P_{1}$ and $B$ can be
choosen as symmetric permutative nonnegative, then $A+XCX^{T}$ becomes
symmetric permutative nonnegative. In fact, if $r=3,$ for instance, we have
that%
\begin{eqnarray*}
&&\left[ 
\begin{array}{ccc}
P_{1} &  &  \\ 
& P_{1} &  \\ 
&  & P_{1}%
\end{array}%
\right] +\left[ 
\begin{array}{ccc}
\mathbf{x} & 0 & 0 \\ 
0 & \mathbf{x} & 0 \\ 
0 & 0 & \mathbf{x}%
\end{array}%
\right] \left[ 
\begin{array}{ccc}
0 & c & c \\ 
c & 0 & c \\ 
c & c & 0%
\end{array}%
\right] \left[ 
\begin{array}{ccc}
\mathbf{x}^{T} & 0 & 0 \\ 
0 & \mathbf{x}^{T} & 0 \\ 
0 & 0 & \mathbf{x}^{T}%
\end{array}%
\right] \\
&=&\left[ 
\begin{array}{ccc}
P_{1} &  &  \\ 
& P_{1} &  \\ 
&  & P_{1}%
\end{array}%
\right] +\left[ 
\begin{array}{ccc}
0 & c\mathbf{xx}^{T} & c\mathbf{xx}^{T} \\ 
c\mathbf{xx}^{T} & 0 & c\mathbf{xx}^{T} \\ 
c\mathbf{xx}^{T} & c\mathbf{xx}^{T} & 0%
\end{array}%
\right] \\
&=&\left[ 
\begin{array}{ccc}
P_{1} & c\mathbf{xx}^{T} & c\mathbf{xx}^{T} \\ 
c\mathbf{xx}^{T} & P_{1} & c\mathbf{xx}^{T} \\ 
c\mathbf{xx}^{T} & c\mathbf{xx}^{T} & P_{1}%
\end{array}%
\right]
\end{eqnarray*}%
is symmetric permutative nonnegative.
\end{remark}

\noindent In particular, for $n=3,$ we have the following pattern of
symmetric permutative nonnegative matrices $B:$%
\begin{equation*}
i)\text{ }\left[ 
\begin{array}{ccc}
a & b & c \\ 
b & c & a \\ 
c & a & b%
\end{array}%
\right] ,\text{ }ii)\text{ }\left[ 
\begin{array}{ccc}
a & a & b \\ 
a & b & a \\ 
b & a & a%
\end{array}%
\right] ,\text{ }iii)\text{ }\left[ 
\begin{array}{ccc}
a & b & b \\ 
b & a & b \\ 
b & b & a%
\end{array}%
\right] ,
\end{equation*}%
with eigenvalues of the form $\lambda _{1},\lambda _{2},-\lambda _{2},$ \
and \ $\lambda _{1},\lambda _{2},\lambda _{2}.$

\noindent Conditions for cases $ii)$ and $iii)$ are%
\begin{eqnarray*}
ii)\text{ }a &=&\frac{1}{3}(\lambda _{1}+\lambda _{2}),\text{ }b=\frac{1}{3}%
(\lambda _{1}-2\lambda _{2}) \\
iii)\text{ }a &=&\frac{1}{3}(\lambda _{1}+2\lambda _{2}),\text{ }b=\frac{1}{3%
}(\lambda _{1}-\lambda _{2}).
\end{eqnarray*}

\noindent However, except for the case $iii),$ $C=B-diagB$\ need not to be
permutative. Consider the following example:

\begin{example}
Let $\Lambda =\{8,6,3,3,-5,-5,-5,-5\}$ with the partition%
\begin{equation*}
\Lambda _{0}=\{8,6,3,3\},\text{ }\Lambda _{1}=\{-5\},\text{ }\Gamma
_{1}=\{5,-5\}.
\end{equation*}%
We compute the matrices%
\begin{equation*}
P_{1}=\left[ 
\begin{array}{cc}
0 & 5 \\ 
5 & 0%
\end{array}%
\right] ,\text{ \ }B=\left[ 
\begin{array}{cccc}
5 & 2 & \frac{1}{2} & \frac{1}{2} \\ 
2 & 5 & \frac{1}{2} & \frac{1}{2} \\ 
\frac{1}{2} & \frac{1}{2} & 5 & 2 \\ 
\frac{1}{2} & \frac{1}{2} & 2 & 5%
\end{array}%
\right] ,
\end{equation*}%
with spectrum $\Gamma _{1}$ and spectrum $\Lambda _{0}$ and diagonal entries 
$5,5,5,5,$ respectively. Then 
\begin{eqnarray*}
A &=&\left[ 
\begin{array}{cccc}
P_{1} &  &  &  \\ 
& P_{1} &  &  \\ 
&  & P_{1} &  \\ 
&  &  & P_{1}%
\end{array}%
\right] +XCX^{T} \\
&=&\left[ 
\begin{array}{cccccccc}
0 & 5 & 1 & 1 & \frac{1}{4} & \frac{1}{4} & \frac{1}{4} & \frac{1}{4} \\ 
5 & 0 & 1 & 1 & \frac{1}{4} & \frac{1}{4} & \frac{1}{4} & \frac{1}{4} \\ 
1 & 1 & 0 & 5 & \frac{1}{4} & \frac{1}{4} & \frac{1}{4} & \frac{1}{4} \\ 
1 & 1 & 5 & 0 & \frac{1}{4} & \frac{1}{4} & \frac{1}{4} & \frac{1}{4} \\ 
\frac{1}{4} & \frac{1}{4} & \frac{1}{4} & \frac{1}{4} & 0 & 5 & 1 & 1 \\ 
\frac{1}{4} & \frac{1}{4} & \frac{1}{4} & \frac{1}{4} & 5 & 0 & 1 & 1 \\ 
\frac{1}{4} & \frac{1}{4} & \frac{1}{4} & \frac{1}{4} & 1 & 1 & 0 & 5 \\ 
\frac{1}{4} & \frac{1}{4} & \frac{1}{4} & \frac{1}{4} & 1 & 1 & 5 & 0%
\end{array}%
\right]
\end{eqnarray*}%
is symmetric permutative nonnegative with spectrum $\Lambda .$
\end{example}

\noindent We finish this section by given a negative answer the question in 
\cite{Paparella}, that is, we show that there are lists of real numbers
which are the spectrum of a nonnegative matrix, but not the spectrum of a
permutative nonnegative matrix.

\begin{lemma}
There is no permutative nonnegative matrix with spectrum $\Lambda
=\{6,5,1\}. $
\end{lemma}

\begin{proof}
It is clear that $\Lambda =\{6,5,1\}$ is realizable (trivially by $%
A=diag\{6,5,1\}$): For instance, the matrix%
\begin{equation*}
A=\left[ 
\begin{array}{ccc}
3 & 0 & 2 \\ 
0 & 6 & 0 \\ 
2 & 0 & 3%
\end{array}%
\right] \text{ has the spectrum }\Lambda .
\end{equation*}%
Suppose $P$ is a permutative nonnegative matrix with spectrum $\Lambda $ and
first row $\left( a,b,c\right) .$ Then $a+b+c=6,$ $P\in \mathcal{CS}_{6},$ $%
tr(P)=12.$ We have three cases:\newline

\noindent $i)$ \ The entries of the main diagonal of $P$ are of the form $%
x,x,x.$ Then $x=4$ and the Perfect necessary and sufficient conditions for
the existence of a nonnegative matrix with prescibed spectrum $\Lambda $ and
diagonal entries $4,4,4$ are not satisfied \cite[Theorem $4,$ condition $(4)$%
]{Perfect1}. Therefore, there is no permutative nonnegative matrix with
spectrum $\Lambda .$\newline

\noindent $ii)$ The entries of the main diagonal of $P$ are of the form $%
x,x,y.$ Then $2x+y=12,$ and from de Perfect conditions \cite[Theorem $4,$
condition $(4)$]{Perfect1}, $x\geq 5,$ $y\leq 2$ or $y\geq 5,$ $x\leq \frac{7%
}{2},$ which contradicts $x+y+z=6,$ except for $x=6,$ $y=z=0.$ In this last
case however, the conditions in \cite[Theorem $4$]{Perfect1} are not
satisfied either.\newline

\noindent $iii)$ The entries of the main diagonal of $P$ are of the form $%
x,y,z.$ Then $x+y+z=12$ contradicts $x+y+z=6.$\newline
Thus, $\Lambda $ cannot be the spectrum of a permutative nonnegative matrix
with spectrum $\Lambda .$
\end{proof}

\bigskip

\noindent Observe, however that $\Lambda =\{6,5,1\}$ is the spectrum of the
direct sum of permutative nonnegative matrices%
\begin{equation*}
A=\left[ 
\begin{array}{ccc}
\frac{11}{2} & \frac{1}{2} & 0 \\ 
\frac{1}{2} & \frac{11}{2} & 0 \\ 
0 & 0 & 1%
\end{array}%
\right] .
\end{equation*}

\noindent After this paper was submitted, R. Loewy \cite{Loewy2} showed that
a realizable list of real numbers need not to be the spectrum of a
permutative nonnegative matrix nor the spectrum of a direct sum of
permutative nonnegative matrices.

\section{Permutative matrices with precribed complex spectrum}

\noindent In this section we show that certain lists of complex numbers are
realizable by permutative nonnegative matrices. First we recall the result
of Loewy and London \cite{Loewy}, which solves the NIEP for $n=3:$

\begin{theorem}
\label{LoLo}Let $\Lambda =\{\lambda _{1},\lambda _{2},\lambda _{3}\}$ be a
list of complex numbers. Then $\Lambda $ is the spectrum of a nonnegative
matrix if and only if 
\begin{equation*}
\begin{array}{ccc}
\Lambda = & \overline{\Lambda },\text{ \ }\lambda _{1}\geq \left\vert
\lambda _{j}\right\vert ,\text{ \ }j=2,3,\text{ \ }\lambda _{1}+\lambda
_{2}+\lambda _{3} & \geq 0 \\ 
& (\lambda _{1}+\lambda _{2}+\lambda _{3})^{2}\leq 3(\lambda
_{1}^{2}+\lambda _{2}^{2}+\lambda _{3}^{2}). & 
\end{array}%
\end{equation*}
\end{theorem}

\noindent Then we have the following

\begin{corollary}
\label{Co1}Every realizable list $\Lambda =\{\lambda _{1},a+bi,a-bi\}$ of
complex numbers is in particular realizable by a permutative nonnegative
matrix.
\end{corollary}

\begin{proof}
The realizing matrix in the proof of Theorem \ref{LoLo} is circulant. Since
circulant matrices are permutative the result follows.
\end{proof}

\begin{corollary}
Let $\Lambda =\{\lambda _{1},\lambda _{2},\ldots \lambda _{n}\}$ be a list
of complex numbers. If there exists a partition $\Lambda =\Lambda _{1}\cup
\cdots \cup \Lambda _{t},$ where $\Lambda _{i}=\{\lambda _{i1},\lambda
_{i2},\lambda _{i3}\},$ $i=1,2,\ldots ,t,$ satisfies conditions of Theorem %
\ref{LoLo}, then $\Lambda $ is the spectrum of a direct sum of permutative
nonnegative matrices with spectrum $\Lambda .$
\end{corollary}

\begin{proof}
The proof is immediate from Corollary \ref{Co1}.
\end{proof}

\bigskip

\begin{remark}
Theorem \ref{So3} can also be applied to a list of complex numbers, as the
following example shows:
\end{remark}

\begin{example}
Let $\Lambda =\{3,2,1,-1\pm i,-1\pm i,-1\pm i\}$ with%
\begin{equation*}
\Lambda _{0}=\{3,2,1\},\ \Gamma _{1}=\{2,-1+i,-1-i\}.
\end{equation*}%
The matrices%
\begin{equation*}
P_{1}=\left[ 
\begin{array}{ccc}
0 & 1-\frac{\sqrt{3}}{3} & 1+\frac{\sqrt{3}}{3} \\ 
1+\frac{\sqrt{3}}{3} & 0 & 1-\frac{\sqrt{3}}{3} \\ 
1-\frac{\sqrt{3}}{3} & 1+\frac{\sqrt{3}}{3} & 0%
\end{array}%
\right] ,\text{ \ }B=\left[ 
\begin{array}{ccc}
2 & 0 & 1 \\ 
0 & 2 & 1 \\ 
0 & 1 & 2%
\end{array}%
\right]
\end{equation*}%
are permutative with spectrum $\Gamma _{1}$ and $\Lambda _{0},$
respectively. Moreover $B$ has the required diagonal entries. Then%
\begin{equation*}
P=\left[ 
\begin{array}{ccc}
P_{1} &  &  \\ 
& P_{1} &  \\ 
&  & P_{1}%
\end{array}%
\right] +XCX^{T}=\left[ 
\begin{array}{ccc}
P_{1} & 0 & \frac{1}{3}\mathbf{ee}^{T} \\ 
0 & P_{1} & \frac{1}{3}\mathbf{ee}^{T} \\ 
0 & \frac{1}{3}\mathbf{ee}^{T} & P_{1}%
\end{array}%
\right]
\end{equation*}%
is permutative nonnegative with spectrum $\Lambda .$
\end{example}

\noindent Next we recall that an $n\times n$ circulant matrix is a matrix of
the form%
\begin{equation*}
C=\left[ 
\begin{array}{ccccc}
c_{0} & c_{1} & c_{2} & \cdots & c_{n-1} \\ 
c_{n-1} & c_{0} & c_{1} & \ddots & \vdots \\ 
\ddots & c_{n-1} & \ddots & \ddots & c_{2} \\ 
c_{2} & \ddots & \ddots & \ddots & c_{1} \\ 
c_{1} & c_{2} & \cdots & c_{n-1} & c_{0}%
\end{array}%
\right] ,
\end{equation*}%
and it is uniquely determined by the entries of its first row, which we
denoted by $\mathbf{c}=\left( c_{0},c_{1},\ldots ,c_{n-1}\right) .$ It is
clear that $C$ is also permutative. Let $\mathbf{\lambda }=\left( \lambda
_{1},\lambda _{2},\ldots ,\lambda _{n}\right) $ with%
\begin{eqnarray*}
\lambda _{1} &=&c_{0}+c_{1}+\cdots +c_{n-1} \\
\lambda _{j} &=&c_{0}+c_{1}\omega ^{j-1}+c_{2}\omega ^{2(j-1)}+\cdots
+c_{n-1}\omega ^{(n-1)(j-1)},\text{ }j=2,\ldots ,n, \\
\omega &=&\exp \left( \frac{2\pi i}{n}\right) ,
\end{eqnarray*}%
being the eigenvalues of the circulant matrix $C=circ\left(
c_{0},c_{1},\ldots ,c_{n-1}\right) .$ Then%
\begin{equation}
\lambda _{n-j+2}=\overline{\lambda _{j}},\text{ \ }j=2,3,\ldots ,\left[ 
\frac{n+1}{2}\right] .  \label{pj}
\end{equation}%
Let 
\begin{eqnarray*}
F &=&\left( f_{kj}\right) =\left[ \mathbf{1}\mid \mathbf{v}_{2}\mid \cdots
\mid \mathbf{v}_{n}\right] \text{ \ with} \\
\mathbf{1} &\mathbf{=}&\left( 1,1,\ldots ,1\right) ^{T} \\
\mathbf{v}_{j} &=&\left( 1,\omega ^{j-1},\omega ^{2(j-1)},\ldots ,\omega
^{(n-1)(j-1)}\right) ^{T},\text{ }j=2,\ldots ,n.
\end{eqnarray*}%
Then%
\begin{eqnarray}
f_{kj} &=&\omega ^{(k-1)(j-1)},\text{ \ }1\leq k,j\leq n,\text{ \ }F%
\overline{F}=\overline{F}F=nI,  \notag \\
F\mathbf{c} &=&\mathbf{\lambda }\text{ \ and \ }\mathbf{c=}\frac{1}{n}%
\overline{F}\mathbf{\lambda .}  \label{F}
\end{eqnarray}%
The following result shows that a list of complex Suleimanova type, with the
property (\ref{pj}), is realizable by a permutative nonnegative matrix.

\begin{theorem}
\label{RLS}Let $\Lambda =\{\lambda _{1},\lambda _{2},\ldots ,\lambda _{n}\}$
be a list of complex numbers with%
\begin{equation*}
\lambda _{j}\in \{z\in 
\mathbb{C}
:\func{Re}z\leq 0,\text{ }\left\vert \func{Re}z\right\vert \geq \left\vert 
\func{Im}z\right\vert \},\text{ \ }j=2,3,\ldots ,n,
\end{equation*}%
satisfying $\lambda _{n-j+2}=\overline{\lambda _{j}},$ \ $j=2,3,\ldots ,%
\left[ \frac{n+1}{2}\right] .$ Then $\Lambda $ is the spectrum of a
permutative nonnegative matrix if and only if $\dsum\limits_{j=1}^{n}\lambda
_{j}\geq 0.$
\end{theorem}

\begin{proof}
The condition is necessary. Suppose $\dsum\limits_{j=1}^{n}\lambda _{j}\geq
0.$ Let $\mu =-\dsum\limits_{j=2}^{n}\lambda _{j}.$ Then the list $\Lambda
_{\mu }=\{\mu ,\lambda _{2},\ldots ,\lambda _{n}\}$ is realizable. From (\ref%
{F}) we have that

\begin{equation}
c_{k}=\frac{1}{2m+1}\left( 
\begin{array}{c}
\mu +2\dsum\limits_{j=2}^{m+1}\func{Re}\lambda _{j}\cos \frac{2k(j-1)\pi }{%
2m+1}+ \\ 
+2\dsum\limits_{j=2}^{m+1}\func{Im}\lambda _{j}\sin \frac{2k(j-1)\pi }{2m+1}%
\end{array}%
\right) ,  \label{ckk}
\end{equation}%
$k=0,1,\ldots ,2m$ \ if $n=2m+1,$ and%
\begin{equation*}
c_{k}=\frac{1}{2m+2}\left( 
\begin{array}{c}
\mu +2\dsum\limits_{j=2}^{m+1}\func{Re}\lambda _{j}\cos \frac{k(j-1)\pi }{m+1%
}+(-1)^{k}\lambda _{m+2}+ \\ 
+2\dsum\limits_{j=2}^{m+1}\func{Im}\lambda _{j}\sin \frac{k(j-1)\pi }{m+1}%
\end{array}%
\right) ,
\end{equation*}%
$k=0,1,\ldots ,2m+1$ \ if $n=2m+2.$ Let $n=2m+1.$ Then%
\begin{equation}
c_{k}=\frac{1}{2m+1}\left( 
\begin{array}{c}
2\dsum\limits_{j=2}^{m+1}{\LARGE (}\cos \frac{2k(j-1)\pi }{2m+1}-1{\LARGE )}%
\func{Re}\lambda _{j}+ \\ 
+2\dsum\limits_{j=2}^{m+1}\func{Im}\lambda _{j}\sin \frac{2k(j-1)\pi }{2m+1}%
\end{array}%
\right) .  \label{ckk1}
\end{equation}%
Since $\func{Re}\lambda _{j}\leq 0,$ then ${\LARGE (}\cos \frac{2k(j-1)\pi }{%
2m+1}-1{\LARGE )}\func{Re}\lambda _{j}\geq 0,$ $j=2,\ldots ,m+1.$ Moreover Im%
$\lambda _{j}\sin \frac{2k(j-1)\pi }{2m+1}\geq 0,$ $j=2,\ldots ,m+1,$ $%
k=0,1,\ldots ,2m.$ Then 
\begin{equation*}
c_{k}\geq 0,\text{ \ }k=0,1,\ldots ,2m,
\end{equation*}%
and $C=circ\left( c_{0},c_{1},\ldots ,c_{n-1}\right) $ is a circulant
nonnegative matrix with spectrum $\Lambda _{\mu }$, which is also
permutative nonnegative. The proof is similar for $n=2m+2.$ If $%
\dsum\limits_{j=1}^{n}\lambda _{j}=\alpha >0,$ then $\alpha =\lambda
_{1}-\mu ,$ and the matrix $C^{\prime }=C+\frac{\alpha }{n}\mathbf{ee}^{T}$
is circulant nonnegative (permutative nonnegative) with spectrum $\Lambda .$
\end{proof}

\end{document}